\documentclass[12pt]{amsart}

\usepackage{amsthm}
\usepackage{amssymb}
\usepackage{amsmath}

\usepackage{amsfonts}
\usepackage{graphicx}
\usepackage{a4wide}

\newtheorem{theorem}{Theorem}[section]
\newtheorem{lemma}[theorem]{Lemma}
\newtheorem{remark}[theorem]{Remark}
\numberwithin{equation}{section}

\def\zz{\mathbb{Z}}
\def\rr{\mathbb{R}}

\def\bn{\mathbf{n}}

\def\rr{\mathbb{R}}

\def\what{\widehat}

\def\eps{\epsilon}

\def\mt{\mathbb T}
\def\bvarphi{\boldsymbol{\varphi}}

\def\bu{\mathbf{u} }

\def\bv{\mathbf{v} }

\def\hat{\what}
\begin{document}

\title{On the numerical approximations of the periodic Schr\"odinger
equation
}

\author{Liviu I. Ignat}

\address{L. I. Ignat
\hfill\break\indent Institute of Mathematics ``Simion Stoilow'' of the Romanian Academy\\
\hfill\break\indent  21 Calea Grivitei Street \\010702 Bucharest \\ Romania 
}
 
 \email{{\tt
liviu.ignat@gmail.com}\hfill\break\indent  {\it Web page: }{\tt
http://www.imar.ro/\~\,lignat}}




\abstract{We consider  semidiscrete finite differences schemes for the periodic
Scr\"odinger equation in dimension one. We  analyze whether the space-time integrability
properties  observed by Bourgain in the continuous case are satisfied at the numerical
level uniformly with respect to the mesh size. For the simplest finite differences scheme
we show that, as mesh size tends to zero, the blow-up in the
 $L^4$ time-space norm occurs, a phenomenon due to the  presence of numerical spurious high
frequencies.

To recover the uniformity of this property we introduce two methods: a spectral filtering
of initial data and a viscous scheme. For both of them we prove a $L^4$ time-space
estimate, uniform with respect to the mesh size.}

\vspace{1cm}

\textbf{Warning 2019}: This paper was submitted to M2AN in 2007 and it was assigned the number 
2007-29. It passed a first review round ( three reviews :-) )  without decision   ("It is only after the consideration of a thoroughly revised
version of your manuscript and a new iteration with all the referees that  I will be in position to make my final decision.")
 
After completing the PhD, I tried to publish the papers resulting from the thesis and I didn't spent too much time on other related problems (the periodic case for example). It was only recently that I discovered some interest in the subject  from other authors and I decided to upload it on arxiv.org. 

Use with caution as there is no revision of the text in the last twelve years.

\endabstract
 \maketitle

\section{Introduction}

Let us consider the linear Schr\"odinger equation(LSE):
 \begin{equation}\label{linear.sch}
 iu_t+\Delta u=0.
\end{equation}
When this equation is considered on the whole space, in addition to
the conservation of energy:
\[\|u(t)\|_{L^2(\rr^d)}=\|u(0)\|_{L^2(\rr^d)}\] we have a dispersive
property:
 \begin{equation}\label{linfty}
\|u(t)\|_{L^\infty(\rr^d)}\leq \frac {c(d)}{
|t|^{d/2}}\|u(0)\|_{L^1(\rr^d)},\, t\neq 0.
\end{equation}
These properties have been employed to develop well-posedness results for inhomogeneous
and nonlinear Schr\"odinger equations (NSE) \cite{0372.35001, MR801582, 0638.35021}. The
main idea of these works  is to obtain space-time estimates for the solutions of LSE,
called Strichartz estimates \cite{0372.35001}:
\begin{equation}\label{strichartz}
\|u\|_{L^q(\rr\times\rr^d)}\leq c(d,q) \|u(0)\|_{L^2(\rr^d)}, \,
q=\frac{2(d+2)}d.
\end{equation}

The local existence theory for NSE in $\rr^d$ uses the dispersive properties of the free
Schr\"odinger semigroup, in the form of the Strichartz estimates. When the domain is
periodic, i.e. when the equation \eqref{linear.sch} is considered on the torus
$\mt^d=\rr^d/2\pi\zz^d$, inequality \eqref{strichartz} does not hold. Using refined
properties of trigonometric series, Bourgain \cite{MR1209299} has obtained an analogue of
these estimates, by defining the temporal norm on a finite time interval and dealing with
the projection of the linear Schr\"odinger equation on spaces spanned by a finite number
of Fourier modes. In dimension one, this leads to estimates similar to classical
Strichartz estimates in the case of the Cauchy problem, except that they involves
constantes that increase with the number of the considered Fourier modes:
$$\Big\|\sum _{|k|\leq N}a_k e^{-4it\pi^2 k^2}e^{2i\pi k  x}\Big\|_{L^6(0,1/2\pi;\,\mt^1)}\leq c N^\eps
\Big (\sum _{|k|\leq N}|a_k|^2\Big)^{1/2},$$ for some positive
constant $\eps$.
 However, for $2\leq q\leq 4$ an estimate similar
to \eqref{strichartz} holds:
\begin{equation*}\label{str.tor}
\|u\|_{L^q(0,1/2\pi;\,\mt^1)}\leq c(q) \|u(0)\|_{L^2(\mt^1)},
\end{equation*}
 where the above constant does not depends by the number
of Fourier modes of initial data.
 The case $q=4$, is a simple consequence of the orthogonality of the  complex  exponentials
 and goes back to Zygmund
\cite{MR0387950}:
\begin{equation}\label{l4}
\Big\|\sum _{|k|\leq N}a_ke^{-4it\pi^2 k^2}e^{2i\pi k
x}\Big\|_{L^4(0,1/2\pi;\,\mt^1)}\leq c \Big (\sum _{|k|\leq N}|a_k|^2\Big)^{1/2}.
\end{equation}

 Using those estimates, Bourgain \cite{MR1209299} has been proved that NSE
   on the one-dimensional torus
$iu_t+\Delta u=|u|^\sigma u$ is locally well-posed in $H^s(\mt^1)$, provided that
$\sigma<2/(1-2s)$. If $\sigma<1$, it is globally posed in the space $L^4(\rr_{loc}\times
\mt^1)$ for initial data in $L^2(\mt^1)$. If $\sigma=1$, the solution is in
$C(\rr,H^s(\mt^1))$ for all $u(0)\in H^s(\mt)$, $s\geq 0$.

The extension of \eqref{strichartz} for the Schr\"odinger flow  to  riemannian compact
manifolds has been recently studied in \cite{MR2058384}, where the authors have been
obtained Strichartz estimates with loose of derivatives. More precisely, for any $p\geq
2$ and $q<\infty$ satisfying the admissibility condition
$$\frac 2p+\frac dq=\frac d2$$
the solutions of the Schr\"odinger equation on the $d$-dimensional
torus $\mt ^d$ satisfy
\begin{equation}\label{stric-burq}
\|u\|_{L^p(I,\,L^q(\mt^d))}\leq C(I)\|u(0)\|_{H^{1/p}(\mt^d)}.
\end{equation}

The aim of this paper is to analyze whether the property \eqref{l4} holds for
approximations of the one-dimensional LSE with periodic boundary conditions. We will
consider two numerical schemes for the linear Schr\"odinger equation on the
one-dimensional torus and analyze whether the space-time estimates of the type \eqref{l4}
hold at the discrete level, uniformly with respect to the mesh size parameter. The
analysis of property (\ref{stric-burq}) at the discrete level will be done in a further
paper. The existence of uniform estimates in mixed norms for the discrete solutions will
permit to introduce convergent numerical approximations for those NSE for which even in
the continuous case the well posedness could not be established by energy methods, see
\cite{MR1691575} and the references therein. To determine all the pairs $(q,r)$ for which
the discrete solutions remain uniformly bounded in the $L^qL^r$ norm when initial data is
bounded in the $L^2$-norm as the mesh size goes to zero, is a difficult open problem.

In the case of the LSE on the whole line, the analysis of the Strichartz  estimates for
its numerical approximations has been done in \cite{1063.35016, liv3} and
\cite{liv-siam}. The authors have proven the lack of uniform dispersive properties of
type \eqref{linfty} or \eqref{strichartz} for the solutions of the simplest approximation
of the linear Schr\"odinger equation:
\begin{equation}\label{first.diss}
i\bu_t+\Delta_h\bu=0,
\end{equation}
where uniformity is with respect to the mesh size. Above, $\Delta_h$
is the second order approximation by finite differences of the
Laplace operator $\Delta$:
\[(\Delta_h \bu)_{j}=\frac{ u_{j+1}-2u_j+u_{j-1}}{h^2},\, j\in \zz.\]

The lack of a uniform estimate of type \eqref{linfty}  is due to the
fact that the symbol $p_h(\xi)=4/h^2\sin(\xi h/2)$ of the operator
$-\Delta_h$ changes the convexity at the points $\pm\pi/2h$, a
property that the continuous one $\xi^2$, does not satisfy.
Observing this pathology, in \cite{liv3}  the following estimate for
the solutions of
 scheme \eqref{first.diss} is proved:
\[\|\bu(t)\|_{l^\infty(h\zz^d)}\leq c(d) \Big(\frac 1{|t|^{1/2}}+
\frac 1{|th|^{1/3}}\Big)\|\bu(0)\|_{l^1(h\zz^d)},\] estimate that is
not uniform on the mesh parameter $h$. This does not allow to prove
uniform Strichartz-like estimates for the  semi-discretization
\eqref{first.diss}.

We mention that the  Schr\"odinger equation on the lattice $h\zz^d$,
without concern for the uniformity of the estimates with respect to
the size of the lattice, has been also studied in \cite{MR2150357}.
The analysis of dispersive properties for fully discrete models has
been done in \cite{MR2000069} for the KdV equation and in
\cite{liv-m3as} for the Schr\"odinger equation.

In the case of the approximations of the periodic LSE, we will prove in Section
\ref{sec.cons} that the same pathology occurs. The symbol introduced by the simplest
finite difference approximation changes the convexity and thus the $L^4$-estimate of the
numerical solution does not hold uniformly with respect to the mesh size.

The paper is organized as follows. In Section \ref{sec.cons} we analyze the simplest
finite difference scheme for the one-dimensional LSE on the torus $\mt^1$. We  prove the
blow-up of the discrete $L^4(0,T;L^4(\mt^1))$-norm of its solutions as mesh parameter
goes to zero. In Section \ref{filter} by mean of Ingham's inequality we show that a
suitable Fourier filtering of the initial data allows us to prove an uniform
$L^4$-estimate on the numerical solutions. Section \ref{vis} is devoted to a scheme that
contains artificial numerical viscosity. In this case the $L^4$-estimate follows by using
the dissipative effect for the high frequency component of the solutions and Ingham's
inequality for the
 low ones.

\section{The conservative scheme}\label{sec.cons}

In this section we will  analyze the semidiscrete scheme in finite differences for the
one-dimensional LSE .
 Let us choose $N$ an even positive integer and
$h=1/(N+1)$. With the convention $u^h_{-1}=u^h_{N}$  we consider the
following numerical scheme:
\begin{equation}\label{conservative}
 \left\{
\begin{array}{ll}
\displaystyle i\frac {d u^h_j}{dt}+\frac{u^h_{j+1}-u^h_j+u^h_{j-1}}{h^2}= 0, &t>0,\ j=0,\dots,N, \\
\\
u^h_0(t)=u^h_{N+1}(t), & t> 0,
  \end{array}
  \right.
\end{equation}
where $\bu^h(0)$ is an approximation of the initial datum $u(0)$. Here $\bu^h$ stands for
the finite unknown vector $\{u_j^h\}_{j=0}^N$, $u_j^h(t)$ being the approximation of the
continuous solution at the node $x_j=j h$ and time $t$.

This scheme satisfies the classical properties of consistency and
stability which imply $L^2$-convergence. We will construct
 explicit solutions for scheme (\ref{conservative}) for
which an $L^4$-estimate similar to  \eqref{l4} does not hold uniformly with respect to
the mesh-size parameter $h$.

In what follows for any $1\leq p<\infty$, we  consider the spaces of
sequences $\bu=(u_0,\dots,u_N)$, $L^p(\mt_h)$,  endowed with the
norm
$$\|\bu\|_{L^p(\mt_h)}^p=h\sum _{j=0}^N |u_j|^p.$$


The main ingredient in our analysis is the \textit{discrete Fourier transform} (DFT) (for
refined results on this transform we refer to \cite{MR1776072}, Ch.~3), which associate
to any vector $\bv=(v_0,v_1,\dots,v_N)$ its Fourier coefficients:
 \begin{equation*}\label{fourier}
 \hat v(k)=h\sum _{j=0}^{N} v_je^{-2i\pi k jh}, \, k=-\frac
N2,\dots, \frac N2.
\end{equation*}
 Also we can recover $\textbf{v}$ from
$\hat v$ by the \textit{inverse discrete Fourier transform} as
follows:
$$v_j=\sum_{k=-N/2}^{N/2} \hat v(k)e^{2i\pi kjh}.$$
In the sequel  we will denote by $\{\bvarphi^h_k\}_{k=-N/2}^{N/2}$, $-N/2\leq k\leq N/2$,
the following vectors
$\bvarphi^h_k=(\varphi^h_{kj})_{j=0}^N$ where $\varphi^h_{kj}=\exp(2ik\pi j h)$. 
These vectors corresponds to the projection of the functions
$\exp(2i\pi kx)$ on the grid $\{jh: j=0,\dots,N\}$.

\begin{figure}
  \includegraphics[width=10cm]{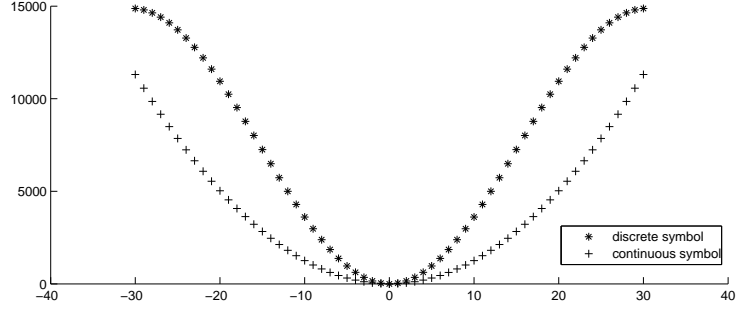}\\
  \caption{The discrete and continous symbols, $h=1/60$}\label{fig1}
\end{figure}

Let  $\{\hat{u}_0(n)\}_{n=-N/2}^{N/2},$  be the Fourier coefficients of $\bu^h(0)$. Thus
$\bu^h(0)$  writes as follows:
    $$\bu^h(0)=\sum _{n=-N/2}^{N/2}\hat{u}_0(n) \bvarphi^h_n.$$
Introducing this representation in equation \eqref{conservative} we find that the
coefficients $\{\hat{u}(t,n)\}$, $-N/2\leq n\leq N$, of $\bu^h(t)$ satisfy the following
ODE's:
$$\left\{
\begin{array}{ll}
 \displaystyle i\frac{d\hat{u}(n,t)}{dt}+\hat{u}(n,t)
 \frac{\exp(2i\pi nh)+\exp(-2i\pi nh)-2}{h^2}=0, & t>0, \\
  \\
  \hat u(n,0) =\hat{u}_0(n).&
\end{array}
\right.$$
   Solving these ODE's we
    obtain that the solution $\bu^h$ of \eqref{conservative} is given by
\begin{equation}\label{psch7}
    \bu^h(t)=\sum _{n=-N/2}^{N/2}\hat{u}_0(n)\exp\Big(-\frac{4it\sin^2(n\pi h)}{h^2}\Big)\bvarphi^h_n
\end{equation}

Observe that the in contrast with the continuous case, where the
solutions are given by the following expression
$$u(t,x)=\sum _{n\in \zz}\hat u_0(n) \exp(-4i\pi^2 n^2 t)\exp(2i\pi nx),$$
here the symbol $n^2$, $n\in \zz$, has been  replaced by
\begin{equation}\label{symbol}
p_h(n)= \frac{4\sin^2\big (n\pi h\big)}{h^2}= 4{(N+1)^2}\sin^2\Big(\frac{n\pi}{N+1}
\Big), \, n=-\frac N2,\dots,\frac N2,
\end{equation}
 and the vectors $\bvarphi_n^h$ are the projections of the
exponentials $\exp(2i\pi kx)$ on the considered grid.

As we can see in Figure \ref{fig1} the discrete symbol $p_h(n)$ introduced by  scheme
\eqref{conservative} changes convexity at the points $n=\pm (N +1)/4$. We recall that, in
the case of Cauchy problem for LSE, this pathology of the discrete symbol has been used
in \cite{1063.35016} to prove the lack of uniform Strichartz estimates for the classical
numerical approximation \eqref{first.diss}. It is then natural to expect that also in the
periodic case the space time integrability of the solutions will  be loosed in its
simplest approximation scheme.

The following theorem shows the existence of initial data in $L^2$ for which the mixed
$L^4$-norm of the approximate solutions blows-up as the mesh size tends to zero.

\begin{theorem}
\label{explosion} Let be $T>0$. Then the following holds:
\begin{equation}\label{exp.1}
    \sup _{h>0}\sup _{\bu^h(0)\in L^2(\mt_h)} \frac {\|\bu^h\|_{L^4(0,T;L^4( \mt_h))}}
    {\|\bu^h(0) \|_{L^2(\mt_h)}}=\infty,
\end{equation}
where  $\bu^h$ is the solution of equation \eqref{conservative} with
$\bu^h(0)$ as initial datum.
\end{theorem}

\begin{remark}
Let $\mathbf{P}^h$ be the piecewise constant interpolator. In view
of Theorem \ref{explosion}, for any
    fixed $T>0$, the uniform boundedness principle guarantees the
existence of a function $u(0)\in L^2(\mathbb{T^1})$ and a sequence $\bu^h(0)$ such that
$\mathbf{P}^h\bu^h(0)\rightarrow u(0)$ in $L^2(\mathbb{T^1})$ such that the corresponding
solutions $\bu^h$ of (\ref{conservative}) satisfy
$$\|\mathbf{P}^h \bu^h \|_{L^4(0,T;L^4( \mt^1))}\rightarrow \infty.$$
\end{remark}

In what follows, to avoid the presence of constants, we will use the
notation $A\lesssim B$ to report the inequality $A\leq constant
\times B$, where the constant is independent of $h$. The statement
$A\simeq B$ is equivalent to $A\lesssim B$ and $B\lesssim A$.

\begin{proof}[Proof of Theorem \ref{explosion}]Let us fix $T>0$.
We will consider initial data with their spectrum concentrated at
the point $N/4$. A similar construction can be done for $-N/4$, this
point having the same pathology as the previous one.

Let us fix $\alpha\in (0,1)$. We set
 $$\Lambda _N=\Big\{n: |n|\leq \frac N2, \, \Big|n-\frac {N}4\Big|<N^\alpha\Big\}.$$
 To prove \eqref{exp.1} we choose as initial datum in
problem \eqref{conservative} the following function:
\begin{equation}\label{initial.data}
\bu^h(0)=\sum _{n\in \Lambda_N}\bvarphi^h_n.
\end{equation}
This function has all its Fourier coefficients identically one in $\Lambda_N$ and vanish
outside this set. A similar construction has been done in \cite{liv3} in the context of
the analysis of the dispersive properties of the classical finite difference
approximation of the Cauchy problem for the Schr\"odinger equation. By choosing initial
data concentrated at the pathological points the solutions of the scheme will be more and
more discrepant to its continuous counterpart.

\begin{figure}
  \includegraphics[width=8cm]{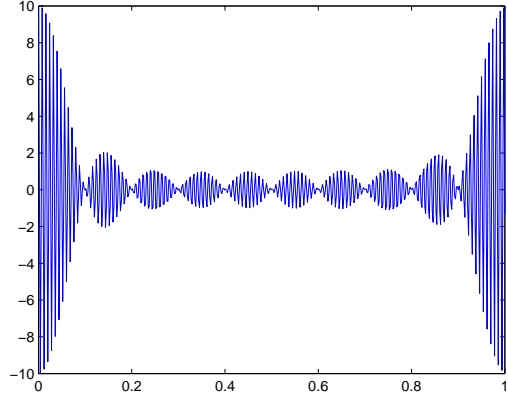}\\
  \caption{Example of initial data chosen  in \eqref{initial.data}, N=500, $\alpha=1/4$ }\label{fig2}
\end{figure}

Using the orthogonality of the vectors $\{\bvarphi^h_n\}_{n=-N/2}^{N/2}$ we obtain  the
following behaviour of the $L^2$-norm of the initial datum:
\begin{equation}\label{norma.initiala}
    \|\bu^h(0)\|_{L^2(\mt_h)} \simeq N^{\alpha/2}.
\end{equation}
We will prove that the $L^4(0,T;L^4(\mt_h))$-norm of
\begin{equation}\label{sol.part}
\bu^h(t)=\sum _{n\in \Lambda_N}\exp\Big(-\frac{4it\sin^2(n\pi h)}{h^2}\Big)\bvarphi^h_n,
\end{equation}
solution of equation \eqref{conservative} with initial datum given by
\eqref{initial.data},
 increases faster than
$N^{\alpha/2}$. To evaluate the $L^4(0,T;L^4(\mt_h))$-norm of $\bu^h$ we will  use that
for any vector $\bv$ given by $$\bv=\sum _{n=-N/2}^{N/2}\hat{a}(n)\bvarphi^h_n$$ its
$L^4(\mt_h)$-norm  satisfies
\begin{equation}\label{psch500}
    \|\bv\|^4_{L^4(\mt_h)}=\sum
    _{\begin{subarray}{c}
-N/2\leq n_i\leq N/2 \\
     n_1+n_2=n_3+n_4
     \end{subarray}
     }
     \hat{a}(n_1)\hat{a}({n_2})\overline{\hat{a}}(n_3)
     \overline{\hat{a}}(n_4)
\end{equation}
The proof of \eqref{psch500} uses only the orthogonality of the vectors $\bvarphi^h_n$
and we will omit it. Applying this result to the solution $\bu^h$ written in the form
\eqref{psch7} we get
$$\|\bu^h\|_{L^4(0,T;\mt_h)}^4=\int _0^T\sum _{\begin{subarray}{c}
n_i\in \Lambda_N \\
     n_1+n_2=n_3+n_4
     \end{subarray}
     } e^{-itq_h(\bn)}dt,$$
where the function $q_h(\bn)$ is given by
\begin{equation}\label{def.pgrande}
q_h(\bn)=p_h(n_1)+p_h(n_2)-p_h(n_3)-p_h(n_4),\, \bn=(n_1,n_2,n_3,n_4).
\end{equation}
  Using that the
function $q_h$ satisfies
$q_h(n_1,n_2,n_3,n_4)=-q_h(n_3,n_4,n_1,n_2)$ we have
\begin{align*}\label{psch13}
    \|\bu^h\|_{L^4(0,T;\,L^4(\mt_h))}^4&=\frac 12\sum _{
    \begin{subarray}{c}
    n_i\in \Lambda_N\\
    n_1+n_2=n_3+n_4
    \end{subarray}
    }
    \int _0^T
    (e^{-itq_h(n_1,n_2,n_3,n_4)}+e^{-itq_h(n_3,n_4,n_1,n_2)})dt\\
    &=\sum _{
    \begin{subarray}{c}
    n_i\in \Lambda_N\\
    n_1+n_2=n_3+n_4
    \end{subarray}
    }
    \int _0^T \cos(t q_h(\bn))dt
    \end{align*}
We will prove that for large enough $N$, the term $q_h(\bn)$
occurring in the right hand side is sufficiently small and then all
the integrals occurring in the last sum behaves as $T$.

Explicit computations show that for any $\bn=(n_1,n_2,n_3,n_4)$ with
$n_1+n_2=n_3+n_4$ the following holds:
\begin{align*}
 q_h(\bn)=8(N+1)^2\cos&\Big(\frac { (n_1+n_2)\pi}{N+1}\Big)
 \sin\Big(\frac {(n_1-n_2+n_3-n_4)\pi} {2(N+1)}\Big)
    \sin\Big(\frac { (n_1-n_2-n_3+n_4)\pi}{2 (N+1)}\Big).
\end{align*}

We estimate each  term  in the above representation of $q_h(\bn)$. Using that $n_1$ and
$n_2$ belong to $\Lambda_N$ we get
$$\left|\frac {(n_1+n_2)\pi}{N+1}-\frac\pi 2\right|\leq
\left|\frac { n_1\pi}{N+1}-\frac\pi 4\right|+\left|\frac {n_2\pi }{N+1}-\frac\pi 4\right|
\leq\frac {2\pi(N^\alpha+1/4)}{(N+1)}\lesssim \frac {1}{(N+1)^{1-\alpha}}.$$ Thus
\begin{equation}\label{ineg.1}
\left|\cos\Big(\frac { (n_1+n_2)\pi}{N+1}\Big)\right|\lesssim \sin \Big(\frac
1{(N+1)^{1-\alpha}}\Big) \lesssim \frac 1{(N+1)^{1-\alpha}}.
\end{equation}
Also using that $|n_1-n_2|+|n_3-n_4|\leq 4(N+1)^\alpha$ we obtain that
\begin{equation}\label{ineg.2}
    \left|\sin\Big(\frac {(n_1-n_2+n_3-n_4)\pi} {2(N+1)}\Big)
    \sin\Big(\frac { (n_1-n_2-n_3+n_4)\pi}{2
    (N+1)}\Big)\right|\lesssim \frac{1}{(N+1)^{2-2\alpha}}.
\end{equation}
In view of inequalities \eqref{ineg.1} and \eqref{ineg.2} we obtain that $q_h(\bn)$
satisfies  $|q_h(\bn)|\lesssim N^{3\alpha-1}$  for any $\bn=(n_1,n_2,n_3,n_4)\in
\Lambda_N^4$ with $n_1+n_2=n_3+n_4$. Choosing $\alpha <1/3$ the following estimate holds
uniformly
 for all $\bn$ as above:
$$\int _0^T \cos(t q_h(\bn))dt \geq T\cos\Big(\frac T {N^{1-3\alpha}}\Big)\gtrsim 1.$$

Using that the number of pairs $\bn\in \Lambda_N^4$ with $n_1+n_2=n_3+n_4$ is of order
$|\Lambda _N|^3$ we obtain that that
\begin{equation}\label{norma.finala} \|\bu^h\|_{L^4(0,T;\,L^4(\mt_h))} \simeq
N^{3\alpha/4}.
\end{equation}
By \eqref{norma.initiala} and \eqref{norma.finala} we finally obtain
that
$$\frac{\|\bu^h\|_{L^4(0,T;\,L^4(\mt_h))}}{\|\bu^h(0)\|_{l^2(\mt_h)}}\simeq
N^{\alpha/4}$$ which finishes the proof.
\end{proof}

%

\section{Filtering the high frequencies}\label{filter}

In this section we prove that a spectral filtering of the initial data will provides
uniform $L^4$-estimates for the discrete solutions. The following Theorem shows that this
holds for all initial  data supported far away from the pathological spectral points $\pm
N/4$.

\begin{theorem}\label{t.fil}
Let  $\lambda <1/4$. There is a positive constant $c(\lambda)$ such that for
 all initial
datum $\bu^h(0)\in {\rm{Span}}\{\varphi _n^h: |n|\leq \lambda N\}$, the solution $\bu ^h$
of equation (\ref{conservative}) satisfies
\begin{equation}\label{psch81}
    \|\bu^h\|_{L^4(0,T;\,L^4( \mathcal \mt_h))}\leq
    c(\lambda)\|\bu^h(0)\|_{L^2(\mt_h)}.
    \end{equation}
\end{theorem}

\begin{remark}
The same result holds if for some positive constant $\eps$ the
 initial datum $\bu^h(0)$ belongs to the set
$${\rm{Span}}\Big\{\varphi _n^h: \Big||n|-\frac N4\Big|\geq \epsilon N\Big\}.$$
On this set the symbol introduced by the finite difference scheme \eqref{conservative} is
uniformly convex, with a parameter which depends by $\eps$. Clearly, this uniform
convexity is lost as $\eps$ goes to zero.
\end{remark}

The main ingredient in the proof of the above result is the following \textit{direct}
Ingham's inequality.
\begin{lemma}
Let $(\mu_n)_{n\in \zz}$ be a sequence of real numbers and $\gamma >
0$ be such
\begin{equation}\label{gap}
\mu_{n+1}-\mu_n\geq \gamma>0,\, \forall n\in \zz.
\end{equation}

For any $T > 0$ there exists a positive constant $C = C(T, \gamma)
> 0$ such that, for any finite sequence $(a_n)_{n\in Z}$,
\begin{equation}\label{ingham}
    \int_{-T}^T\Big|\sum _{n}a_ne^{i\mu_n t}\Big|^2dt\leq
    C\sum_{n}|a_n|^2.
\end{equation}
\end{lemma}

The Ingham's  inequalities have been successfully used to prove observability
inequalities in for many 1-D control problems. 
They generalize the classical Parseval
equality for orthogonal sequences. We refer to Young \cite{MR1836633} to a survey on the
theme.

\begin{proof}[Proof of Theorem \ref{t.fil}]
 Let us consider $\bu^h(0)\in {\rm{Span}}\{\varphi _n^h: |n|\leq \lambda
N\}$ as initial data in equation \eqref{conservative}:
$$\bu^h(0)=\sum _{|n|\leq \lambda N} \hat{u}_0(n) \bvarphi^h_n.$$
The solution of \eqref{conservative} is given by
$$ \bu^h(t)=\sum _{|n|\leq \lambda N}
\hat{u}_0(n)\exp(-it p_h(n))\bvarphi^h_n$$ where $p_h(n)= 4/{h^2}\sin^2 (n\pi h)$.

 With $q_h$ as in \eqref{def.pgrande}, identity  \eqref{psch500} applied to vector
 $\bu^h$ gives us
\begin{align*}
   \|\bu^h\| &^4_{L^4(0,T;\, L^4(\mt_h)}=
   \int _0^T \sum _
    {\begin{subarray}{c}
    |n_i|\leq \lambda N\\
    n_1+n_2=n_3+n_4
    \end{subarray}
    }
   \hat{u}_0(n_1) \hat{u}_0(n_2)\overline{\hat{u}}_0(n_3) \overline{\hat{u}}_0(n_4)
   e^{-itq_h(\bn)}
 \\
   &=\int _0^T \sum _{|r|\leq 2\lambda N }\sum _
    {\begin{subarray}{c}
     |n_i|\leq \lambda N\\
    n_1+n_2=r,\,n_3+n_4=r
    \end{subarray}
    }
   \hat{u}_0(n_1) \hat{u}_0(n_2)\overline{\hat{u}}_0(n_3) \overline{\hat{u}}_0(n_4)
   e^{-itq_h(\bn)}\\
    &=\sum _{|r|\leq 2\lambda N}\int _0^T\Big|\sum _
    {\begin{subarray}{c}
     |n_i|\leq \lambda N\\
    n_1+n_2=r
    \end{subarray}
    }\hat{u}_0(n_1) \hat{u}_0(n_2)e^{-it(p_h(n_1)+p_h(n_2))}\Big|^2dt.
     \end{align*}
It is sufficient to prove that for any $|r|\leq 2\lambda N$, $\Sigma(r)$ defined by
\begin{equation*}\label{psch82}
\Sigma (r)=\int _0^T\Big|\sum _
    {\begin{subarray}{c}
    |n_i|\leq \lambda N\\
    n_1+n_2=r
    \end{subarray}
    }\hat{u}_0(n_1) \hat{u}_0(n_2)e^{-it(p_h(n_1)+p_h(n_2))}\Big|^2dt
\end{equation*}
satisfies \begin{equation}\label{estimare.sigma} \Sigma (r)\leq
    \sum _
    {\begin{subarray}{c}
    |n_i|\leq \lambda N\\
    n_1+n_2=r
    \end{subarray}
    }|\hat{u}_0(n_1) |^2|\hat{u}_0(n_2)|^2.
\end{equation}

We consider the case $r\geq 0$ the other case being similar. We
write $\Sigma (r)$ as
\begin{equation}\label{psch83}
    \Sigma(r) = \int _0^T\Big|\sum _{r-\lambda N\leq n\leq \lambda N}
     \hat{u}_0(n) \hat{u}_0(r-n)e^{it\mu_h(n)}\Big|^2dt
\end{equation}
where $\mu_h(n)=-p_h(n)-p_h(r-n)$.

\begin{figure}
  \includegraphics[width=8cm]{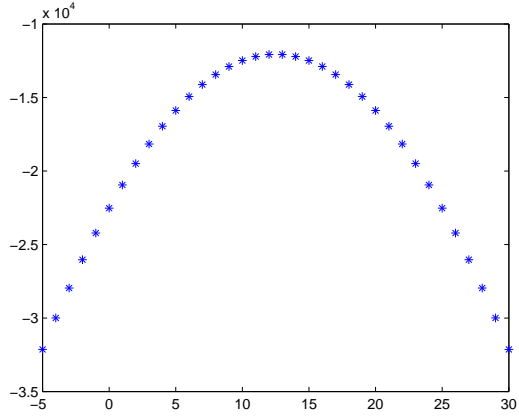}\\
  \caption{The sequence $\mu_h(n)$, $N=150$, $\lambda=1/5$, $r=N/6$, $n=r-\lambda N,\dots,\lambda N$}
  \label{fig3}
\end{figure}

In order to apply Ingham's inequality \eqref{ingham} we need to prove that the
 sequence $\mu_h(n)$  satisfies the gap condition (\ref{gap}). As we can
see in Figure \ref{fig3}, there are two range of $n$'s for which the gap condition
(\ref{gap}) is satisfied in each of them.

Elementary manipulations of trigonometric functions give us that
\begin{equation}\label{psch84}
  \mu_h(n+1)-\mu_h(n)=8(N+1)^2\cos\Big(\frac {\pi r}{N+1}\Big)\sin
\Big(\frac{\pi(r-2n-1)}{N+1}\Big)\sin \Big(\frac {\pi}{N+1}\Big).
\end{equation}
Thus the sequence $\mu_h(n)$ satisfies the following inequalities:
$$\mu_h(r-[\lambda N])< \mu_h(r-[\lambda n]+1)<\dots < \mu _h\Big(\Big[\frac r2\Big]-1\Big)
< \mu_h\Big(\Big[\frac r2\Big]\Big) $$ and
$$\mu_h\Big(\Big[\frac r2\Big]+1\Big)>\mu_h\Big(\Big[\frac r2\Big]+2\Big)>\dots> \mu_h(\lambda N),$$
where $[\cdot]$ represents the floor function.

By \eqref{psch84}, for any $n\leq [r/2]-1$ in the definition \eqref{psch83} of
$\Sigma(r)$ we obtain
$$\mu_h(n+1)-\mu_h(n)\geq 8(N+1)^2\cos(2\lambda \pi)\sin^2\Big(\frac \pi {N+1}\Big)\geq c_1(\lambda).$$
Also, for $n\geq [r/2]$ the following holds
$$\mu_h(n+1)-\mu_h(n)\leq -8(N+1)^2\cos(2\lambda \pi)\sin^2\Big(\frac \pi {N+1}\Big)\leq -c_2(\lambda).$$
This suggests us to split the right hand side of \eqref{psch83} as
$\Sigma(r)=\Sigma_1(r)+\Sigma_2(r)$ where
$$\Sigma_1(r)=\int _0^T\Big|\sum _{n=r-\lambda N}^{[r/2]}
     \hat{u}_0(n)\hat{u}_0({r-n})e^{it\mu_h(n)}\Big|^2dt$$
and
$$\Sigma_2(r)=\int _0^T\Big|\sum _{n=[r/2]+1}^{\lambda N}
     \hat{u}_0(n)\hat{u}_0({r-n})e^{it\mu_h(n)}\Big|^2dt.$$

For each of the above terms we apply Ingham's inequality and obtain:
$$\Sigma_1(r)\leq \sum _{r-\lambda N\leq n}^{[r/2]}|\hat{u}_0(n)|^2|\hat{u}_0(r-n)|^2$$
and
$$\Sigma_2(r)\leq \sum _{[r/2]+1}^{\lambda N}|\hat{u}_0(n)|^2|\hat{u}_0(r-n)|^2$$
The last two estimates show that \eqref{estimare.sigma} holds and
then \eqref{psch81}. The proof is now complete.
\end{proof}

\section{A viscous scheme}\label{vis}

In this section we will analyze the $L^4$-property for numerical schemes that contain
artificial numerical viscosity.

The scheme we will analyze is  as follows:
\begin{equation}\label{sch.diss}
    \left\{
\begin{array}{lc}
\displaystyle i\frac{du_j}{dt}
+\frac{u^h_{j+1}-2u^h_j+u^h_{j-1}}{h^2}=ia(h)\frac{u^h_{j+1}-2u^h_j+u^h_{j-1}}{h^2},
& j=0,\dots, N,\, t>0,\\
\\
u_0(t)=u_{N+1}(t),&t>0,\\
\end{array}
    \right.
\end{equation}
where by convention $u_{-1}=u_{N}$.

The main result of this section is given by the following theorem.
\begin{theorem}
Let  $a(h)$ be such that
\begin{equation}\label{conditie.a}
\inf_{h>0}\frac{a(h)}h>0.
\end{equation}
For any positive time $T$, there exists a positive constant $C(T)$ such that
\begin{equation}\label{l4l4}
\|\bu^h\|_{L^4(0,T;\,L^4(\mt_h))}\leq C(T)\|\bu^h(0)\|_{L^2(\mt _h)}
\end{equation} holds for all $\bu^h(0)\in L^2(\mt_h)$, uniformly in $h>0$.
\end{theorem}

\begin{remark}
We do not know whether  condition \eqref{conditie.a} on  $a(h)$ is necessary. It is an
open problem to determine the range of $\beta>1$, if there exists any, such that for
$a(h)=h^\beta$ the result of the above theorem still holds.
\end{remark}

\begin{proof}

Tacking the DFT in \eqref{sch.diss} we find that the Fourier coefficients of $\bu^h(t)$,
$\hat u(t,k)$, $k=-N/2,\dots,N/2$ solve the following ODE's:

$$\left\{\begin{array}{ll}
\displaystyle i\frac{d\hat u(t,k)}{dt}+\hat u(t,k)p_h(n)=ia(h)
 \hat u(t,k)p_h(n), \, t >0.\\
 \\
\hat{u}(0,k)=\hat u_0(k),
 \end{array}\right.$$
where $\hat u_0(k), k=-N/2,\dots,N/2$ are the Fourier coefficients of $\bu^h(0)$:
$$\bu^h(0)=\sum_{|n|\leq N/2}\hat u_0(n)\bvarphi_n^h.$$
Solving the above ODE's we obtain that $\bu^h$, the solution of \eqref{sch.diss} is as
follows
$$\bu^h(t)=\sum _{k=-N/2}^{N/2} \hat u_0(k)\exp
\left(-\frac {4it\sin^2(k\pi h)}{h^2}-a(h)\frac{4t\sin^2(k\pi h)}{h^2} \right)
\bvarphi_k^h.$$

We split  $\bu^h$ in two components $\bu^h(t)=\bu^h_{low}(t)+\bu^h_{high}(t),$
corresponding to a low:
\begin{equation}\label{ulow}
\bu^h_{low}(t)=\sum _{|k|\leq N/8} \hat u_0(k)\exp\left(-itp_h(n)-a(h)tp_h(n) \right)
\bvarphi_k^h,
\end{equation}
 respectively high frequencies component of $\bu^h$:
\begin{equation}\label{uhigh}
\bu^h_{high}(t)=\sum _{N/8<|k|\leq N/2} \hat u_0(k)\exp\left(-itp_h(n)-a(h)tp_h(n)
\right) \bvarphi_k^h.
\end{equation}
The choice of $N/8$ as the range of low frequencies is motivated by
the fact that there are solutions concentrated near the point $N/4$
of the spectrum that generate the blow-up of the $L^4(0,T;\,L^4(\mt
_h))$-norm as we have seen in Section \ref{sec.cons}.

 In the
following we will prove that the two components $\bu^h_{low}$ and $\bu^h_{low}$  of
$\bu^h$  satisfy:
$$\|\bu^h_{low}\|_{L^4(0,T;\,L^4(\mt_h))}\lesssim \|\bu^h_{low}(0)\|_{L^2(\mt_h)}$$
and
$$\|\bu^h_{high}\|_{L^4(0,T;\,L^4(\mt_h))}\lesssim \|\bu^h_{high}(0)\|_{L^2(\mt_h)}.$$

 {\bf Step I. Estimate of $\bu_{high}^h$.} Using identity  \eqref{psch500} 
we obtain that $\bu^h_{high}$ satisfies the rough estimate
\begin{align*}
  \|\bu_{high}^h\|^4_{L^4(0,T;\,L^4(\mt_h))} &\leq   \sum
    _{\begin{subarray}{c}
N/8< |n_i|\leq N/2 \\
     n_1+n_2=n_3+n_4
     \end{subarray}
     }|\hat{u}_0(n_1)
     \hat{u}_0(n_2)\hat{u}_0(n_3)\hat{u}_0(n_4)|\int _0^T e^{-ta(h)\sigma_h(\bf n)}dt\\
   &\leq \sum
    _{\begin{subarray}{c}
N/8< |n_i|\leq N/2 \\
     n_1+n_2=n_3+n_4
     \end{subarray}
     } \frac {|\hat{u}_0(n_1)
     \hat{u}_0(n_2)\hat{u}_0(n_3)\hat{u}_0(n_4)|}{a(h)\sigma_h(\bf
     n)}.
\end{align*}
where $\sigma_h$ is given by
 $$\sigma_h(\bn)=p_h(n_1)+p_h(n_2)+p_h(n_3)+p_h(n_4), \, \bn=(n_1,n_2,n_3,n_4).$$

Observe that for all $\textbf{n}=(n_1,n_2,n_3,n_4)$ with $N/8<n_i\leq N/2 $,
$\sigma_h(\bn)$ can be bounded from below as follows:
$$\sigma_h(\bn)\geq {16}(N+1)^2\sin^2 \Big(\frac{N\pi}{8(N+1)}\Big)\geq c_1(N+1)^2.$$
Thus, in the above inequality we replace $\sigma_h(\bn)$ by $(N+1)^2$ and we obtain
\begin{equation}\label{ineg.5}
   \|\bu_{high}^h\|^4_{L^4(0,T;\,L^4(\mt_h))} \lesssim \frac {h^2}{a(h)}\sum
    _{\begin{subarray}{c}
N/8< |n_i|\leq N/2 \\
     n_1+n_2=n_3+n_4
     \end{subarray}
     }|\hat{u}_0(n_1)
     \hat{u}_0(n_2)\hat{u}_0(n_3)\hat{u}_0(n_4)|.
\end{equation}
The right hand side sum satisfies:
\begin{align}
\nonumber\sum
    _{\begin{subarray}{c}
N/8< |n_i|\leq N/2 \\
     n_1+n_2=n_3+n_4
     \end{subarray}
     }    |\hat{u}_0(n_1)
    \hat{u}_0(n_2)&\hat{u}_0(n_3)\hat{u}_0(n_4)|
  \leq\sum _{N/4< r\leq N } \Big(\sum
    _{\begin{subarray}{c}
N/8< |n_i|\leq N/2 \\
     n_1+n_2=r
     \end{subarray}
     }|\hat{u}_0(n_1) \hat{u}_0(n_2)|\Big)^2\\
    \nonumber &\lesssim \sum _{N/4< r\leq N } (r+1)\sum
    _{\begin{subarray}{c}
N/8< |n_i|\leq N/2 \\
     n_1+n_2=r
     \end{subarray}
     }|\hat{u}_0(n_1) \hat{u}_0(n_2)|^2
  \label{ineg.6} \\
  &  \lesssim (N+1) \Big(\sum_{N/8\leq n\leq N/2} |\hat{u}_0(n)|^2
     \Big)^2.
\end{align}
Estimates (\ref{ineg.5}) and (\ref{ineg.6}) imply that
$$\|\bu_{high}^h\|^4_{L^4(0,T;\,L^4(\mt_h))} \lesssim
\frac h{a(h)}{\|\bu_{high}^h(0)\|_{L^2(\mt_h)}^4}.$$ In view of assumption
\eqref{conditie.a} on $a(h)$ we obtain that the high frequencies component of $\bu^h$,
$\bu^h_{high}$, satisfies the following estimate:
$$\|\bu_{high}^h\|_{L^4(0,T;\,L^4(\mt_h))} \lesssim \|\bu_{high}^h(0)\|_{L^2(\mt_h)}.$$

{\bf Step II. Estimate of $\bu_{low}^h$.} Using the same ideas as above the low
frequencies component $\bu_{low}^h$ satisfies:
\begin{align*}
   \|\bu_{low}^h\|^4_{L^4(0,T;\,L^4(\mt_h))} &=
   \int _0^T \sum _
    {\begin{subarray}{c}
    |n_i|\leq N/8\\
    n_1+n_2=n_3+n_4
    \end{subarray}
    }
   \hat{u}_0(n_1)
     \hat{u}_0(n_2)\overline{\hat{u}}_0(n_3)\overline{\hat{u}}_0(n_4)
   e^{-itq_h(\bn)}e^{-ta(h)\sigma_h(\bn)}dt\\
   &=\int _0^T\sum _{|r|\leq N/4}\sum _
    {\begin{subarray}{c}
    |n_i|\leq N/8\\
    n_1+n_2=n_3+n_4=r
    \end{subarray}
    }
   \hat{u}_0(n_1)
     \hat{u}_0(n_2)\overline{\hat{u}}_0(n_3)\overline{\hat{u}}_0(n_4)
   e^{-itq_h(\bn)}e^{-ta(h)\sigma_h(\bn)}dt\\
    &=\int _0^T \sum _{|r|\leq N/4}\Big|\sum _
    {\begin{subarray}{c}
    |n_i|\leq N/8\\
    n_1+n_2=r
    \end{subarray}
    }\hat{u}_0(n_1)
     \hat{u}_0(n_2)e^{-it(p_h(n_1)+p_h(n_2))}e^{-ta(h)(p_h(n_1)+p_h(n_2))}\Big|^2dt.
\end{align*}

It is sufficient to prove the existence of a positive constant $c$, independent of
$|r|\leq N/4$ and $N$, such that $\Sigma(r)$ defined by
\begin{equation}\label{psch92}
\Sigma(r) =\int _0^T\Big|\sum _
    {\begin{subarray}{c}
    |n_i|\leq N/8\\
    n_1+n_2=r
    \end{subarray}
    }\hat{u}_0(n_1)
     \hat{u}_0(n_2)e^{-it(p_h(n_1)+p_h(n_2))}e^{-ta(h)(p_h(n_1)+p_h(n_2))}\Big|^2dt,
\end{equation}
satisfies  the following inequality:
\begin{equation}\label{est.puta}
\Sigma(r)
    \leq
    c
\sum _
    {\begin{subarray}{c}
    |n_i|\leq N/8\\
    n_1+n_2=r
    \end{subarray}} |\hat{u}_0(n_1)|^2|
     \hat{u}_0(n_2)|^2.
\end{equation}

We consider the case $r\geq 0$, the other case cans be treated in a similar manner. The
conditions imposed on $r,n_1,n_2$  in  definition \eqref{psch92} of  $\Sigma(r)$ give us
that $n_1$ satisfies:
$$-\frac {N}8\leq r-\frac{N}8\leq n_1\leq \frac{N}8\leq r+\frac {N}8. $$
This allows us to rewrite  $\Sigma(r)$   in the following form
$$\Sigma(r)=\int _0^T\Big|\sum _
    {n=\,r-N/8}^{ N/8}\hat{u}_0(n)
     \hat{u}_0(r-n)e^{-it(p_h(n)+p_h(r-n))}e^{-ta(h)(p_h(n)+p_h(r-n))}\Big|^2dt.$$
Using that  $r-{N}/8\leq  r/2\leq  {N}/8$  we write $\Sigma(r)\leq
2(\Sigma_1(r)+\Sigma_2(r))$ where
$$\Sigma_1(r)=\int _0^T\Big|\sum _
    {n=\,r-N/8}^{ [r/2]}\hat{u}_0(n)
     \hat{u}_0(r-n)e^{-it(p_h(n)+p_h(r-n))}e^{-ta(h)(p_h(n)+p_h(r-n))}\Big|^2dt$$
and
$$\Sigma_2(r)=\int _0^T\Big|\sum _
    {n=[r/2]+1}^{ N/8}\hat{u}_0(n)
     \hat{u}_0(r-n)e^{-it(p_h(n)+p_h(r-n))}e^{-ta(h)(p_h(n)+p_h(r-n))}\Big|^2dt.$$

In the following  we  prove that
$$\Sigma_1(r)\lesssim C(T)\sum _
    {n=r-N/8}^{ [r/2]}|\hat{u}_0(n)|^2 |\hat{u}_0(r-n)|^2.$$
In a similar manner $\Sigma_2(r)$ will satisfy
$$\Sigma_2(r)\lesssim C(T) \sum _
    {n=[r/2]+1}^{ N/8}|\hat{u}_0(n)|^2 |\hat{u}_0(r-n)|^2$$
and thus \eqref{est.puta} holds, which finishes the proof.

With the notations $\mu_h(n,m)=q_h(n,r-n,m,r-m)$ and $\nu_h(n,m)=\sigma_h(n,r-n,m,r-m)$
we get
\begin{eqnarray*}
  \Sigma_1(r) &\leq &\sum _{n=\,r-N/8}^{ [r/2]}|\hat{u}_0(n)|^2
   |\hat{u}_0(r-n)|^2\int _0^Te^{-ta(h)\nu_h(n,m)}dt+S_1(r)\\
    &\leq &  T\sum _
    {n=r-N/8}^{ [r/2]}|\hat{u}_0(n)|^2 |\hat{u}_0(r-n)|^2+S_1(r).
\end{eqnarray*}
where $$S_1(r)=\sum _{\begin{subarray}{c}
    n,m=r- N/8\\
    n\neq m
    \end{subarray}
    }^{ [r/2]}|\hat{u}_0(n)
     \hat{u}_0(r-n) {\hat{u}}_0(m){\hat{u}}_0(r-m)|
  \left|  \int _0^T e^{-it\mu_h(n,m)}e^{-ta(h)\nu_h(n,m)}dt\right|$$
It is sufficient  to prove that
\begin{equation}\label{est.S1}
S_1(r)\lesssim\sum _
    {n=r-N/8}^{ [r/2]}|\hat{u}_0(n)|^2 |\hat{u}_0(r-n)|^2.
\end{equation}
Observe that  $S_1$  satisfies the following  estimate :
\begin{align*}
    S_1(r) &\leq 2 \sum _{\begin{subarray}{c}
    n,m=r- N/8\\
    n\neq m
    \end{subarray}
    }^{ [r/2]}
    \frac{ |\hat{u}_0(n)
     \hat{u}_0(r-n){\hat{u}}_0(m)
     {\hat{u}}_0(r-m)|}{|i\mu_h(n,m)+a(h)\nu_h(n,m)|}
     \\
     &= \sum _
    {n=r-N/8}^{ [r/2]}|\hat{u}_0(n)|^2 |\hat{u}_0(r-n)|^2 \sum _{\begin{subarray}{c}
    m=r- N/8\\
    m\neq n
    \end{subarray}
    }^{ [r/2]}\frac 1{|\mu_h(m,n)|}\\
    &\qquad+\sum _
    {m=r-N/8}^{ [r/2]}|\hat{u}_0(m)|^2 |\hat{u}_0(r-m)|^2 \sum _{\begin{subarray}{c}
    n=r- N/8\\
    n\neq m
    \end{subarray}
    }^{ [r/2]}\frac 1{|\mu_h(m,n)|}.
\end{align*}
We prove the existence of a positive constant $C$, independent of $N$ and $r$ such that
the following
$$l(n)=\sum _{\begin{subarray}{c}
    m=r- N/8\\
    m\neq n
    \end{subarray}
    }^{ [r/2]}\frac 1{|\mu_h(m,n)|}\leq C$$
holds for all $n=r-N/8,\dots,[r/2]$. Thus  \eqref{est.S1}  holds and the proof is
finishes.

 We claim the existence of a positive constant $c$
such that $\mu_h$ verifies
\begin{equation}\label{est.p}
|\mu_h(n,m)|\geq c|n-m||r-m-n|
\end{equation}
 for all $n,m\in I(r,N)=\{r-N/8,\dots,[ r/2]\}$.
 Thus for a fixed $n$, $l(n)$ satisfies
\begin{align*}
 l(n)
\lesssim \sum _{\begin{subarray}{c}
    m=r- N/8\\
    m\neq n
    \end{subarray}
    }^{ [r/2]} \frac 1{|n-m|^2}+\sum _{\begin{subarray}{c}
    m=r- N/8\\
    m\neq n
    \end{subarray}
    }^{ [r/2]} \frac 1{|r-m-n|^2} \lesssim  \sum _{m\geq 1}\frac 1{m^2}\leq C.
\end{align*}

It remains to  prove \eqref{est.p}. Explicit computations show that
$$\mu_h(n,m)=8(N+1)^2\cos\Big(\frac{r\pi}{N+1}\Big)\sin\Big(\frac{(r-m-n)\pi}{N+1}
\Big)\sin\Big(\frac{(n-m)\pi}{N+1}\Big).$$ Taking into account that $r\leq N/4$ we get
$|\cos(r\pi/(N+1))|\geq \cos(\pi/4) $. This was the key point in choosing the low
frequency component up to $N/8$. If we would have chosen all the
 frequencies up to $N/4$ we will obtain  that
$r\leq N/2$ and then the term $\cos(r\pi/N)$ could not be bounded from bellow by a
positive constant. Using that $n$ and $m$ belong to $I(r,N)$ and  that $r$ is nonnegative
we have $|n|\leq N/8$ and $|m|\leq N/8$. Thus $|n-m|\leq N/4 $ and
$$\Big|\sin\Big(\frac{(n-m)\pi}{N+1}\Big)\Big|\gtrsim \frac{|n-m|}{N+1} . $$
Also $r\leq N/4$ implies that  $|r-m-n|\pi/(N+1)\leq \pi/2$ and
$$\Big|\sin\Big(\frac{(r-m-n)\pi}{N+1}
\Big)\Big|\gtrsim \frac{|r-n-m|}{N+1}.$$ The last  two inequalities
prove \eqref{est.p}.

The proof is now complete.
\end{proof}

%


\end{document}